\documentclass[11pt]{article}

\usepackage[a4paper,margin=29mm]{geometry}
\usepackage[T1]{fontenc}
\usepackage{lmodern}
\usepackage{microtype}
\usepackage{amsmath,amssymb,amsthm,mathtools}
\usepackage{enumitem}
\usepackage{booktabs,tabularx,array}
\usepackage{xcolor}
\usepackage{hyperref}
\usepackage{fancyhdr}

\definecolor{linkblue}{HTML}{205493}
\hypersetup{colorlinks=true,linkcolor=linkblue,citecolor=linkblue,urlcolor=linkblue}
\setlist{nosep,leftmargin=2em}
\newtheorem{theorem}{Theorem}[section]
\newtheorem{lemma}[theorem]{Lemma}

\newtheorem{proposition}[theorem]{Proposition}
\newtheorem{conjecture}[theorem]{Conjecture}
\theoremstyle{definition}

\theoremstyle{remark}
\newtheorem*{remark}{Remark}

\newcommand{\Ztwo}{\mathbb Z/2}
\newcommand{\CP}{\mathbb{CP}}
\newcommand{\orb}{\mathrm{orb}}

\title{A Seifert Dichotomy for $\Ztwo$-Harmonic One-Forms}
\author{Weifeng Sun\\
\texttt{sunweifeng8@mail.tsinghua.edu.cn}}
\date{}

\begin{document}
\maketitle

\begin{abstract}
Let $Y$ be a closed oriented Seifert fibered rational homology $3$-sphere with a fixed
Seifert fibration $\pi:Y\to B$.  Suppose first that the base orbifold is
oriented, so that $B=S^2(\alpha_1,\ldots,\alpha_r)$.  For any fixed
connection-metric family on $Y$, we prove a sharp threshold: if
$r\le3$, certain metrics admit no nonzero
$\Ztwo$-harmonic one-form.  When $r\ge4$, existing
results give existence for every Riemannian metric.  We also complete the
classification for nonorientable bases.
These results give a complete Seifert
dichotomy and motivate a broader conjecture for non-Seifert rational homology
spheres.
\end{abstract}

\section{Introduction}

A $\Ztwo$-harmonic one-form is a closed and coclosed one-form with values in
a flat real line bundle away from its zero set; locally its representatives
may differ by sign.  Such objects arise in gauge-theoretic compactness
problems and are related to harmonic maps to $\mathbb R$-trees and the
Morgan--Shalen compactification~\cite{Taubes,HWZ}.  Their zero sets may have
codimension two, so ordinary Hodge theory alone does not govern existence.

There is an essential distinction on a rational homology sphere: ordinary
harmonic one-forms always vanish, while a $\Ztwo$-harmonic one-form may carry
a nontrivial sign line.  He--Wentworth--Zhang prove that reducible or Haken
closed oriented three-manifolds nevertheless carry such forms for every
metric~\cite[Theorem~1.3 and Corollary~1.4]{HWZ}.  A metric with no nonzero
$\Ztwo$-harmonic one-form is therefore most naturally sought among
irreducible non-Haken rational homology spheres.

Seifert manifolds provide a structured testing ground for this metric versus
topology question.  A connection metric separates the horizontal geometry
of the base orbifold from the circle direction and permits an adiabatic limit
in which the fibers collapse.  On the existence side, square roots of
orbifold quadratic differentials yield invariant $\Ztwo$-harmonic one-forms.
Related Seifert pullback examples were constructed by
Haydys--Mazzeo--Takahashi~\cite{HMT}, and the orbifold construction used here
is closely related to that of He--Parker~\cite{HeParker}.
He--Parker also express the expectation that Brieskorn homology spheres
$\Sigma(a_1,a_2,a_3)$ admit no nonzero $\Ztwo$-harmonic one-forms
\cite[Section~4.3, following Corollary~4.14]{HeParker}.
The main analytic issue is the converse one: to exclude all, not merely
invariant, $\Ztwo$-harmonic one-forms when the base orbifold has too few cone
points.

\paragraph{Geometric setup.}

Let $Y$ be a closed oriented three-manifold equipped with an effective,
locally free circle action $a:S^1\times Y\to Y$.  Write
$a_t(y)=a(t,y)$, and let
\[
  \pi:Y\longrightarrow B=Y/S^1
\]
be the quotient map.  Here \emph{effective} means that no nonidentity element
of $S^1$ acts trivially on all of $Y$, and \emph{locally free} means that the
stabilizer
\[
  (S^1)_y=\{t\in S^1:a_t(y)=y\}
\]
is finite for every $y\in Y$.  The fibers of $\pi$ are precisely the circle
orbits, so $\pi$ is the Seifert fibration determined by the action.  We assume
that the quotient is an oriented orbifold whose underlying topological
surface is the two-sphere.  We write
\[
  B=S^2(\alpha_1,\ldots,\alpha_r),
  \qquad \alpha_i\ge 2,
\]
to mean that $B$ has distinct cone points $p_1,\ldots,p_r$, with cyclic
isotropy group of order $\alpha_i$ at $p_i$, and has no other orbifold
singularities.  More explicitly, a neighborhood of $p_i$ is represented by
an orbifold chart
\[
  D\longrightarrow D/(\mathbb Z/\alpha_i),
  \qquad
  z\longmapsto [z],
\]
where a generator of $\mathbb Z/\alpha_i$ acts on the disk $D\subset\mathbb C$
by $z\mapsto e^{2\pi i/\alpha_i}z$.  The point $p_i$ is the image of an
exceptional Seifert fiber whose stabilizer has order $\alpha_i$.  Thus $r$
counts the exceptional fibers, while
\[
  B_{\mathrm{reg}}=B\setminus\{p_1,\ldots,p_r\}
\]
is the regular part of the base.

We next fix a smooth orbifold Riemannian metric $h$ on $B$.  This means that
$h$ is an ordinary smooth Riemannian metric on $B_{\mathrm{reg}}$ and that,
in every cone chart $D\to D/(\mathbb Z/\alpha_i)$, its pullback extends to a
smooth $\mathbb Z/\alpha_i$-invariant Riemannian metric on the whole disk
$D$.

Identify $S^1$ with $\mathbb R/(2\pi\mathbb Z)$, and let $\xi$ denote the
infinitesimal generator of the circle action.  Fix an $S^1$-invariant
connection one-form $\eta\in\Omega^1(Y)$ normalized by $\eta(\xi)=1$.

For each $\varepsilon>0$, the data $(h,\eta)$ define the connection metric
\[
  g_\varepsilon=\pi^*h+\varepsilon^2\eta^2,
  \qquad
  \eta^2:=\eta\otimes\eta.
\]
The orbifold smoothness of $h$ and the invariance of $\eta$ imply,
through the local Seifert charts, that $g_\varepsilon$ extends smoothly across
the exceptional fibers.

\begin{theorem}[Main theorem]\label{thm:main}
Let $Y$ be a closed oriented Seifert fibered rational homology $3$-sphere with an
oriented Seifert fibration
\[
  \pi:Y\longrightarrow B=S^2(\alpha_1,\ldots,\alpha_r),
\]
and let $g_\varepsilon=\pi^*h+\varepsilon^2\eta^2$ be any fixed
connection-metric family.  If $r\le 3$, there is
$\varepsilon_0=\varepsilon_0(\pi,h,\eta)>0$ such that
$(Y,g_\varepsilon)$ admits no nonzero $\Ztwo$-harmonic one-form whenever
$0<\varepsilon<\varepsilon_0$.
\end{theorem}

\begin{proposition}\label{prop:dichotomy}
In fact, when $r\ge 4$, existing results imply that every Riemannian metric
$g$ on $Y$ admits a nonzero $\Ztwo$-harmonic one-form.  Together with
Theorem~\ref{thm:main}, this completes the Seifert dichotomy.
\end{proposition}

\begin{proof}
Detcherry--Kalfagianni--Sikora prove that a closed Seifert manifold fibering
over $S^2$ with at least four exceptional fibers is
Haken~\cite[Proposition~3.1]{DKS}.  Since $Y$ is a rational homology sphere,
the every-metric existence theorem of He--Wentworth--Zhang then applies and
gives a nontrivial $\Ztwo$-harmonic one-form for every Riemannian metric on
$Y$~\cite[Theorem~1.3]{HWZ}.
\end{proof}

In fact, for the connection-metric family considered above, one has the
following more refined, symmetry-preserving existence statement.

\begin{theorem}
\label{thm:invariant-existence}
Under the geometric setup above, suppose that $r\ge 4$.
Then, for every $\varepsilon>0$, the Riemannian manifold
$(Y,g_\varepsilon)$ admits a nonzero $S^1$-invariant
$\Ztwo$-harmonic one-form.
\end{theorem}

Theorem~\ref{thm:invariant-existence} is not essentially new.  The
corresponding invariant forms are already constructed by
He--Parker~\cite[Proposition~1.13, Lemma~4.12, and
Proposition~4.13]{HeParker}.  Although their statement uses a normalized
constant-curvature connection metric, the same pullback construction is
independent of the fiber scale and gives the formulation above for any fixed
connection-metric family.  A proof of
Theorem~\ref{thm:invariant-existence} is given in
Section~\ref{sec:invariant-existence}.

We also have a corresponding dichotomy in the nonorientable-base case.

\begin{proposition}
\label{prop:nonorientable-dichotomy}
Let $Y$ be a closed oriented Seifert fibered rational homology $3$-sphere with a fixed
Seifert fibration over a closed nonorientable base orbifold.  Then the
underlying surface of the base is $\mathbb{RP}^2$; write
$B=\mathbb{RP}^2(\alpha_1,\ldots,\alpha_s)$, where $s$ is the number of
exceptional fibers.  When $s=0$, let $b\in\mathbb Z$ denote the integer
Seifert invariant (obstruction class) of the circle bundle
$Y\to\mathbb{RP}^2$.  Let $p:\widehat Y\to Y$ be the double cover orienting the vertical
tangent line, and suppose that the metric family is fiber-orientation-adapted,
in the sense that
\[
  p^*g_\varepsilon
  =\widehat\pi^*\widehat h+\varepsilon^2\widehat\eta^2
\]
for fixed data on the induced oriented Seifert fibration
$\widehat\pi:\widehat Y\to\widehat B$.  Then:
\begin{enumerate}[label=\textup{(\roman*)}]
\item if $s=0$ and $b=0$, then
  $Y\cong\mathbb{RP}^3\#\mathbb{RP}^3$, and every Riemannian metric on $Y$
  admits a nonzero $\Ztwo$-harmonic one-form;
\item if either $s=0$ and $b\ne0$, or $s=1$, then there exists
  $\varepsilon_0>0$ such that $(Y,g_\varepsilon)$ admits no nonzero
  $\Ztwo$-harmonic one-form for $0<\varepsilon<\varepsilon_0$;
\item if $s\ge2$, then $Y$ is Haken, and every Riemannian metric on $Y$
  admits a nonzero $\Ztwo$-harmonic one-form.
\end{enumerate}
\end{proposition}

Parts~\textup{(i)} and~\textup{(iii)} are not new.  For~\textup{(i)},
Hatcher's classification identifies the zero-obstruction Seifert fibration
over $\mathbb{RP}^2$ without genuine exceptional fibers with
$\mathbb{RP}^3\#\mathbb{RP}^3$~\cite{Hatcher}.
This manifold is reducible, so the every-metric existence statement follows
from He--Wentworth--Zhang~\cite[Corollary~1.4]{HWZ}.  For~\textup{(iii)},
Detcherry--Kalfagianni--Sikora prove that an oriented closed Seifert manifold
over $\mathbb{RP}^2$ with at least two genuine exceptional fibers is
Haken~\cite[Proposition~3.1]{DKS}; the same corollary of
He--Wentworth--Zhang again gives the asserted form for every Riemannian
metric.  The proof of~\textup{(ii)} is given in
Section~\ref{sec:nonorientable-nonexistence}.

\paragraph{A possible dichotomy for all closed oriented three-manifolds.}
For an arbitrary closed connected oriented three-manifold $M$, the relevant
topological cases can be organized as follows.
\begin{enumerate}[label=\textup{(\roman*)}]
\item If $b_1(M)>0$, Hodge theory gives a nonzero ordinary harmonic one-form
  for every Riemannian metric, hence a nonzero $\Ztwo$-harmonic one-form with
  trivial sign line.
\item If $b_1(M)=0$ and $M$ is reducible or Haken, then
  He--Wentworth--Zhang prove that every Riemannian metric admits a nonzero
  $\Ztwo$-harmonic one-form~\cite[Corollary~1.4]{HWZ}.
\item If $M$ is an irreducible non-Haken Seifert fibered rational homology
  $3$-sphere,
  then the results of this paper produce a Riemannian metric with no nonzero
  $\Ztwo$-harmonic one-form.  For an oriented base this follows from
  Theorem~\ref{thm:main}; for a nonorientable base it follows from
  Proposition~\ref{prop:nonorientable-dichotomy}.
\item By geometrization~\cite{AFW}, the remaining unresolved class consists of closed non-Haken hyperbolic
  rational homology $3$-spheres.
\end{enumerate}
We conjecture the complementary metric-existence statement for the fourth
case.

\begin{conjecture}
\label{conj:nonseifert-form-free}
Every closed connected oriented non-Haken hyperbolic rational
homology three-sphere $M$ admits a Riemannian metric $g$ for which $(M,g)$
has no nonzero $\Ztwo$-harmonic one-form.
\end{conjecture}

Separately, He--Wentworth--Zhang~\cite[Conjecture~1.6]{HWZ}
conjecture that, for every Riemannian metric $g$ on $S^3$,
$(S^3,g)$ admits no nonzero $\Ztwo$-harmonic one-form.

\paragraph{Acknowledgements.}
The author is deeply grateful to Clifford Taubes for his pioneering work
on $\Ztwo$-harmonic forms and spinors, and for his generous guidance and
support when the author was a Ph.D. student.  The author also thanks Siqi He,
Jiahuang Chen, and other colleagues for helpful discussions.  This work
was supported by start-up funding from Tsinghua University.

\paragraph{Disclosure of AI usage.}
This paper was prepared in four stages.

First, the human author formulated the research questions and supplied the
overall direction and main ideas.  GPT-5.6 Sol was then tasked with
developing the proofs, and the proposed proofs were reviewed using the
Rethlas verifier.  This stage consumed approximately 65\% of the weekly
usage allowance of a $5\times$ Pro subscription.  Rethlas reported a pass
in its automated LLM review of the informal proofs.  The process produced
an initial AI-generated draft of roughly 100 pages, which was highly
verbose and difficult to read.

Second, the author reorganized and reconstructed the proof.  When details
blocked this reconstruction, the author asked GPT in a chat interface to
locate and explain the relevant arguments in the initial draft.  Through
this reconstruction, the author checked the argument largely by hand and
concluded that the proof could be made rigorous.

Third, the author rewrote the paper with assistance from Codex.  The author
designed the framework and logical organization, while Codex mainly
assisted with English expression and filling in details under the author's
control.

Fourth, the author refined the wording and presentation sentence by
sentence.

Throughout the entire process, the author read very little of
the initial AI-generated draft directly.  Nevertheless, the author was
greatly encouraged by the outcome of the first stage and, by its end,
thought that the argument should be likely to work.

The human author takes full responsibility for all claims and conclusions
in this paper.

\section{The invariant construction and the main proof strategy}
\label{sec:invariant-existence}

\subsection{Proof of Theorem~\ref{thm:invariant-existence}}

We first recall the analytic objects involved.  A $\Ztwo$-harmonic one-form
on a Riemannian three-manifold $(Y,g)$ is a triple $(Z,I,v)$, where $Z$ is a
closed proper subset, $I\to Y\setminus Z$ is a flat Euclidean real line bundle, and
\[
  v\in\Omega^1(Y\setminus Z;I),
  \qquad d_Iv=0,
  \qquad d_I^*v=0.
  \tag{2.1}
\]
We require $v$ to be nowhere zero on $Y\setminus Z$ and $|v|$ to extend
continuously to $Y$, with zero set exactly $Z$.
In addition, $v$ and $\nabla v$ are locally square-integrable and the standard
zero-growth condition holds: for some $C_0,\delta,r_0>0$,
\[
  \int_{B_\rho(x)}|v|^2\,d\operatorname{vol}_g
  \le C_0\rho^{3+\delta}
  \qquad (x\in Z,\ 0<\rho<r_0).
  \tag{2.2}
\]
This is the formulation of~\cite[Definition~3.1]{HWZ} in terms of the flat
sign line $I$.  We regard $(Z,I,v)$ as a two-valued section of
$T^*Y/\{\pm1\}$.  It is $S^1$-invariant if $Z$ is invariant and this
quotient-valued section is fixed by the Seifert circle action.

The construction used below starts with the following elementary description
of orbifold quadratic differentials.  Let $p_1,\ldots,p_r$ also denote the
marked points on the coarse sphere $|B|\cong\CP^1$.

\begin{lemma}\label{lem:orbifold-quadratic-differentials}
There is a natural isomorphism
\[
  H^0_{\orb}\!\left(B,K_{B,\orb}^{\otimes2}\right)
  \cong
  H^0\!\left(\CP^1,
    K_{\CP^1}^{\otimes2}(p_1+\cdots+p_r)\right).
  \tag{2.3}
\]
Consequently, if $r\ge4$, then
\[
  \dim_{\mathbb C}H^0_{\orb}\!\left(B,K_{B,\orb}^{\otimes2}\right)
  =r-3.
  \tag{2.4}
\]
\end{lemma}

\begin{proof}
Near a cone point of order $\alpha$, let $z$ be the coordinate on the local
covering disk $D$, and let $w=z^\alpha$ be the coordinate on the quotient
$D/(\mathbb Z/\alpha)$.  Invariance forces a
holomorphic orbifold quadratic differential to have the form
\[
  \widetilde q=z^{\alpha-2}H(z^\alpha)(dz)^2
\]
for a holomorphic function $H$.  In the quotient coordinate $w$, this becomes
\[
  q=\frac{H(w)}{\alpha^2w}(dw)^2.
  \tag{2.5}
\]
Thus an orbifold quadratic differential is precisely a quadratic differential
on $\CP^1$ with at most a simple pole at each marked point.  This proves
\textup{(2.3)}.  The line bundle on the right is
$\mathcal O_{\CP^1}(r-4)$, whose space of sections has dimension $r-3$ when
$r\ge4$.
\end{proof}

We can now prove Theorem~\ref{thm:invariant-existence}.

\begin{proof}[Proof of Theorem~\ref{thm:invariant-existence}]
The oriented metric $h$ determines a complex structure on the orbifold $B$.
By Lemma~\ref{lem:orbifold-quadratic-differentials}, choose
\[
  0\ne q\in H^0_{\orb}\!\left(B,K_{B,\orb}^{\otimes2}\right).
\]
Let $\Sigma_q\subset B$ be the finite set where $f(z)=0$ in the local
expression $\widetilde q=f(z)(dz)^2$, using the covering coordinate $z$ at
cone points.  Clearly, $\operatorname{Re}\sqrt q$ is a
$\Ztwo$-harmonic one-form on $B^\circ=B\setminus\Sigma_q$.

We record what happens at a cone point.  In the notation of the preceding
proof, write
\[
  q=\frac{F(w)}{w}(dw)^2,
  \qquad \nu=\operatorname{ord}_0F.
\]
In the covering coordinate $z$, its pullback is
\[
  \widetilde q=\alpha^2z^mG(z)(dz)^2,
  \qquad
  m=\alpha-2+\alpha\nu\ge0,
  \qquad G(0)\ne0.
  \tag{2.7}
\]
If $m=0$, then necessarily $\alpha=2$ and $\nu=0$, and the pullback of the
square-root line to the smooth Seifert chart extends across the exceptional
core.  If $m>0$, the cone point belongs to
$\Sigma_q$, and the square-root line is considered on the punctured
neighborhood.  At every zero of order $m>0$, in the local coordinate $z$
(the covering coordinate at a cone point),
\[
  |\operatorname{Re}\sqrt q|^2=O(|z|^m),
  \qquad
  |\nabla(\operatorname{Re}\sqrt q)|^2=O(|z|^{m-2}).
  \tag{2.8}
\]

Define
\[
  Z_q=\pi^{-1}(\Sigma_q),
  \qquad v_q=\pi^*(\operatorname{Re}\sqrt q).
  \tag{2.9}
\]
Let $I_q$ be the flat sign line associated with $v_q$ on $Y\setminus Z_q$.
The cyclic action in each local Seifert chart acts on the two square roots of
$q$, so $(I_q,v_q)$ descends to the smooth total space away from $Z_q$; the
$m=0$ case above extends across the exceptional core.  If $x$ lies on a zero
fiber of order $m\ge1$, then~\textup{(2.8)} gives
\[
  \nabla v_q\in L^2_{\mathrm{loc}},
  \qquad
  \int_{B_\rho(x)}|v_q|^2\,d\operatorname{vol}_{g_\varepsilon}
  =O(\rho^{m+3}).
  \tag{2.10}
\]
These estimates imply the conditions in~\textup{(2.2)}.

Clearly, $v_q$ is the $\Ztwo$-harmonic one-form that we want.
This proves Theorem~\ref{thm:invariant-existence}.
\end{proof}

\subsection{Outline of the proof of Theorem~\ref{thm:main}}

We now assume that $r\le3$ and outline the proof of
Theorem~\ref{thm:main}.  For $0<\varepsilon\le1$, set
\[
  d\mu=\eta\wedge\pi^*dA_h
       =\varepsilon^{-1}d\operatorname{vol}_{g_\varepsilon},
\]
and introduce the fixed bundle
\[
  E=\underline{\mathbb R}e^0\oplus\pi^*_{\orb}T^*B.
\]
Here $e^0$ denotes the standard unit section of the trivial real line bundle
over $Y$.
Give $E$ the product metric for which $e^0$ is a unit section, and set
$g_{\mathrm{ref}}=\pi^*h+\eta^2$.  We use the fixed split metric connection
$\nabla^{\mathrm{split}}$ on $E$ determined by
\[
  \nabla^{\mathrm{split}}e^0=0,
  \qquad
  \nabla^{\mathrm{split}}_{X^H}(\pi^*\alpha)=\pi^*(\nabla_X^h\alpha),
  \qquad
  \nabla^{\mathrm{split}}_\xi(\pi^*\alpha)=0,
\]
where $\nabla^h$ is the Levi--Civita connection of $h$, $\alpha$ is a base
one-form, and the superscript $H$ denotes horizontal lift with respect to
$\eta$: for a base vector field $X$, its lift $X^H$ satisfies
$\pi_*X^H=X$ and $\eta(X^H)=0$.

Let $\Phi_\varepsilon:E\to T^*Y$ be the fiberwise isometry determined by
$\Phi_\varepsilon(e^0)=\varepsilon\eta$ and by the identity on horizontal
covectors.  The rescaled connection on $E$ is defined by
\[
 \nabla^{\varepsilon,\mathrm{rescaled}}\sigma
 :=\Phi_\varepsilon^{-1}
   \bigl(\nabla^{g_\varepsilon}(\Phi_\varepsilon\sigma)\bigr).
\]
If $(Z_\varepsilon,I_\varepsilon,v_\varepsilon)$ is normalized by
\[
  \int_Y|v_\varepsilon|^2\,
  d\operatorname{vol}_{g_\varepsilon}=1,
\]
define
\[
  \sigma_\varepsilon=\Phi_\varepsilon^{-1}(\sqrt{\varepsilon}\,v_\varepsilon)
  =u_\varepsilon e^0+\beta_\varepsilon.
  \tag{2.11}
\]
Then $\|\sigma_\varepsilon\|_{L^2(d\mu)}=1$.

For a flat Euclidean real line bundle $I$, we use the same notation for
the connections on $E\otimes I$ obtained by tensoring with the flat
connection on $I$, wherever $I$ is defined; the meaning is determined by
context.  Directional subscripts denote covariant differentiation; in
particular, $\nabla^{\mathrm{split}}_\xi\sigma_\varepsilon$ is the derivative
in the fiber direction.

\begin{proposition}
\label{prop:uniform-adiabatic-estimate}
There is a constant $C=C(\pi,h,\eta)$ such that every normalized
$\Ztwo$-harmonic one-form satisfies
\[
  \int_Y\left(|\sigma_\varepsilon|^2+
  |\nabla^{\mathrm{split}}\sigma_\varepsilon|_{g_{\mathrm{ref}}}^2\right)d\mu\le C,
  \qquad
  \int_Y|\nabla^{\mathrm{split}}_\xi\sigma_\varepsilon|^2d\mu
  \le C\varepsilon^2.
  \tag{2.12}
\]
The constants are independent of the zero set $Z_\varepsilon$ and of the
flat sign line $I_\varepsilon$.
\end{proposition}

This proposition is proved in Section~\ref{sec:analytic-proofs}.

The second estimate in~\textup{(2.12)} expresses the basic effect of fiber
collapse: after passage to a limit, only the fiberwise zero mode can remain.
Because the sign lines and zero sets may vary, ordinary vector-bundle
compactness cannot be applied directly.  The required compactness statement
is the following.

\begin{proposition}
\label{prop:adiabatic-limit}
Let $\varepsilon_j\to0$ and let $v_j=v_{\varepsilon_j}$ be normalized
$\Ztwo$-harmonic one-forms.  Write $\sigma_j=\sigma_{\varepsilon_j}$,
with $u_j=u_{\varepsilon_j}$ and $\beta_j=\beta_{\varepsilon_j}$.
After passing to a subsequence, $\sigma_j$ converges
strongly in $L^2(d\mu)$, in the sign quotient $E/\{\pm1\}$, to a nonzero
$S^1$-invariant limit $\sigma_\infty$ of norm one.

More explicitly, choose a regular Seifert neighborhood
$D_0\times S^1$, where $D_0$ is a disk in the base.  Near an exceptional
fiber, instead use the smooth finite cyclic cover $D_0\times S^1$ of
the Seifert neighborhood, with $D_0$ the base uniformizing disk, and
pull back all tensors to this cover, retaining their notation.
In either case write $\eta=dt+A$, with
$t\in\mathbb R/(2\pi\mathbb Z)$ and $A$ a smooth one-form on $D_0$, and set
\[
 F_j:D_0\times\mathbb R
 \longrightarrow D_0\times\bigl(\mathbb R/(2\pi\mathbb Z)\bigr),
 \qquad F_j(x,s)=\bigl(x,[s/\varepsilon_j]_{2\pi}\bigr).
\]
Then $F_j^*g_{\varepsilon_j}=h+(ds+\varepsilon_jA)^2$ converges smoothly
on compact subsets to $h+ds^2$, and
\[
 \sqrt{\varepsilon_j}\,F_j^*v_j\longrightarrow w_\infty
 \quad\hbox{strongly in }L^2_{\mathrm{loc}}(D_0\times\mathbb R;h+ds^2)
\]
in the sign quotient of the cotangent bundle.  The scalar norms satisfy
\[
 \bigl|\sqrt{\varepsilon_j}\,F_j^*v_j\bigr|_{F_j^*g_{\varepsilon_j}}
 \longrightarrow |w_\infty|_{h+ds^2}
 \quad\hbox{uniformly on every compact subset of }D_0\times\mathbb R.
\]
Here $w_\infty$ is the translation-invariant $\Ztwo$-harmonic one-form
for $h+ds^2$ obtained by replacing $e^0$ by $ds$ in $\sigma_\infty$.
On its nonzero locus, local branches have the form
\[
  w_\infty=u_0\,ds+\beta_0,
  \qquad
  d_Bu_0=0,
  \qquad
  d_B\beta_0=d_B^*\beta_0=0.
  \tag{2.13}
\]
The zero set of $\sigma_\infty$ in $Y$ is the inverse image of a finite
subset $\Sigma_\infty\subset B$, and its flat sign line descends to
$B\setminus\Sigma_\infty$.  If $u_0\equiv0$, then
$(\beta_0^{1,0})^{\otimes2}$ extends across $\Sigma_\infty$ to a global
orbifold holomorphic quadratic differential on $B$.
\end{proposition}

This proposition is proved in Section~\ref{sec:analytic-proofs}.

The classification of the limit now contains the decisive use of $r\le3$.
If $u_0\equiv0$, the last assertion of
Proposition~\ref{prop:adiabatic-limit} and
Lemma~\ref{lem:orbifold-quadratic-differentials} place
$(\beta_0^{1,0})^{\otimes2}$ in
\[
  H^0\!\left(\CP^1,
  K_{\CP^1}^{\otimes2}(p_1+\cdots+p_r)\right)
  \cong H^0\!\left(\CP^1,\mathcal O_{\CP^1}(r-4)\right).
  \tag{2.14}
\]
This space is zero when $r\le3$, so $\beta_0=0$, contradicting the norm-one
condition.  Hence $u_0$ is a nonzero parallel section.  It trivializes the
limiting sign line and forces the limiting zero set to be empty.  The
remaining $\beta_0$ is then an ordinary orbifold harmonic one-form on the
orbifold $B$, whose underlying topological space is $S^2$, and therefore
vanishes.  Consequently
\[
  \sigma_\infty=\{\pm c e^0\},
  \qquad c>0.
  \tag{2.15}
\]

Strong $L^2$ convergence alone does not imply that $Z_j=\varnothing$ for all
sufficiently large $j$.  The next proposition uses local uniform convergence
of the norms to exclude zeros, and a Weitzenb\"ock energy estimate to prove
derivative decay.

\begin{proposition}
\label{prop:zero-exclusion-energy}
Under the hypotheses of Proposition~\ref{prop:adiabatic-limit}, suppose that
the limit is given by~\textup{(2.15)}.  Then, for all sufficiently large $j$,
\[
  Z_j=\varnothing.
  \tag{2.16}
\]
Thus $I_j$ is a flat real line bundle on all of $Y$, and
\[
  \bigl\|\,|u_j|-c\,\bigr\|_{L^2(d\mu)}\longrightarrow0,
  \qquad
  \int_Y|\nabla^{\mathrm{split}}\sigma_j|_{g_{\mathrm{ref}}}^2d\mu\longrightarrow0.
  \tag{2.17}
\]
\end{proposition}

This proposition is proved in Section~\ref{sec:analytic-proofs}.

The final input is a lemma.

\begin{lemma}
\label{lem:twisted-spectral-gap}
There is $\lambda_*=\lambda_*(Y,g_{\mathrm{ref}})>0$, where
$g_{\mathrm{ref}}=\pi^*h+\eta^2$, such that, for every nontrivial flat
Euclidean real line bundle $I\to Y$ and every $u\in W^{1,2}(Y;I)$,
\[
  \int_Y|d_Iu|_{g_{\mathrm{ref}}}^2d\mu
  \ge\lambda_*\int_Y|u|^2d\mu.
  \tag{2.18}
\]
\end{lemma}

\begin{proof}
Isomorphism classes of flat Euclidean real lines on $Y$ are identified with
$H^1(Y;\mathbb Z/2)$, which is finite.
For any nontrivial flat line $I$, let $\lambda_1(I)$ be the infimum of
\[
 \frac{\int_Y|d_Iu|_{g_{\mathrm{ref}}}^2d\mu}
      {\int_Y|u|^2d\mu}
\]
over nonzero $u\in W^{1,2}(Y;I)$.  Then $\lambda_1(I)>0$.
Taking the minimum over the finite set of nontrivial holonomy classes gives
$\lambda_*>0$ and proves Lemma~\ref{lem:twisted-spectral-gap}.
If no nontrivial class exists, the assertion is vacuous.
\end{proof}

\begin{proof}[Proof of Theorem~\ref{thm:main}]
If the theorem were false, there would be a sequence
$\varepsilon_j\to0$ carrying nonzero normalized $\Ztwo$-harmonic one-forms.
Propositions~\ref{prop:uniform-adiabatic-estimate} and
\ref{prop:adiabatic-limit}, followed by the calculation~\textup{(2.14)}, show
that every such sequence has the pure vertical limit~\textup{(2.15)}.
Proposition~\ref{prop:zero-exclusion-energy} then gives
$Z_j=\varnothing$ for all large $j$ and
\[
  \|u_j\|_{L^2(d\mu)}\longrightarrow1,
  \qquad
  \|d_{I_j}u_j\|_{L^2(d\mu;g_{\mathrm{ref}})}\longrightarrow0,
  \tag{2.19}
\]
where the first limit follows from the norm-one normalization of $\sigma_\infty$
and the fact that $\displaystyle\int_Yd\mu$ is constant.

The uniform lower bound~\textup{(2.18)}
rules out nontrivial holonomy in~\textup{(2.19)} for all sufficiently large
$j$.  Hence $I_j$ is trivial for all sufficiently large $j$.
For such $j$, the form $v_j$ is an ordinary nonzero harmonic one-form on
$Y$.  This is impossible because $Y$ is a rational homology sphere and hence
$H^1(Y;\mathbb R)=0$.  The contradiction proves
Theorem~\ref{thm:main}.
\end{proof}

\section{Proofs of four propositions}
\label{sec:analytic-proofs}

This section ties up the loose ends.

\subsection{The proof of Proposition~\ref{prop:uniform-adiabatic-estimate}}

Since $\eta$ is $S^1$-invariant and $\eta(\xi)=1$, Cartan's formula gives
\[
 \iota_\xi d\eta=\mathcal L_\xi\eta-d(\eta(\xi))=0.
\]
Thus $d\eta$ is horizontal and $S^1$-invariant, so we can write
\[
  d\eta=\pi^*F,
\]
where $F$ is a fixed orbifold two-form on $B$, represented in each local
uniformizing chart by a smooth two-form invariant under the local group
action.  The vector
$\varepsilon^{-1}\xi$ is the unit vertical vector.

Let $\nabla^\varepsilon=\nabla^{g_\varepsilon}$ denote the Levi--Civita
connection of $g_\varepsilon$, and let $\nabla^h$ denote that of $h$.
For a base one-form $\alpha$, the superscript $\sharp_h$ denotes its
metric-dual vector, characterized by $h(\alpha^{\sharp_h},U)=\alpha(U)$
for every base vector $U$.
\begin{lemma}\label{lem:ricci-components}
Let $\operatorname{Ric}_\varepsilon$ and $\operatorname{Ric}_h$ denote
the Ricci curvature tensors of $g_\varepsilon$ and $h$, respectively.
For invariant horizontal lifts $X,Y$, the following identities hold:
\begin{enumerate}
\item[(i)] The vertical component is
\[
 \operatorname{Ric}_\varepsilon(\varepsilon^{-1}\xi,\varepsilon^{-1}\xi)
 =\frac{\varepsilon^2}{2}|F|_h^2.
\]
\item[(ii)] The horizontal component is
\[
 \operatorname{Ric}_\varepsilon(X,Y)
 =\operatorname{Ric}_h(X,Y)
 -\frac{\varepsilon^2}{2}\langle\iota_XF,\iota_YF\rangle_h.
\]
\item[(iii)] The mixed component is
\[
 \operatorname{Ric}_\varepsilon(\varepsilon^{-1}\xi,X)
 =\frac{\varepsilon}{2}(d_h^*F)(X).
\]
\end{enumerate}
\end{lemma}

\begin{proof}
Work in a local base chart (or a uniformizing chart at a cone point),
and choose an $h$-orthonormal frame $X_1,X_2$ with
$\nabla^hX_i=0$ at the point of calculation.  Base vectors are identified
with their invariant horizontal lifts.  Thus $[\xi,X_i]=0$.
Since $\eta(X_i)=\eta(X_j)=0$ and $d\eta=\pi^*F$,
\[
\begin{aligned}
 F(X_i,X_j)
 &=d\eta(X_i,X_j)\\
 &=X_i(\eta(X_j))-X_j(\eta(X_i))-\eta([X_i,X_j])
 =-\eta([X_i,X_j]).
\end{aligned}
\]
As $\eta(\xi)=1$, the vertical part of the bracket is
$-F(X_i,X_j)\xi$.  Hence
\[
 [X_i,X_j]=[X_i,X_j]^H-F(X_i,X_j)\xi,
\]
where the horizontal term is the lift of the base bracket.

\emph{(i) The vertical component.}  Torsion-freeness and
$[\xi,X_i]=0$ give
\[
 \nabla^\varepsilon_{\varepsilon^{-1}\xi}X_i
 =\nabla^\varepsilon_{X_i}(\varepsilon^{-1}\xi).
\]
We have
\[
\begin{aligned}
 2g_\varepsilon(\nabla^\varepsilon_{X_i}(\varepsilon^{-1}\xi),X_j)
 &=g_\varepsilon(\nabla^\varepsilon_{X_i}(\varepsilon^{-1}\xi),X_j)
 -g_\varepsilon(\nabla^\varepsilon_{X_j}(\varepsilon^{-1}\xi),X_i)\\
 &=-g_\varepsilon(\varepsilon^{-1}\xi,
       \nabla^\varepsilon_{X_i}X_j-\nabla^\varepsilon_{X_j}X_i)\\
 &=-g_\varepsilon(\varepsilon^{-1}\xi,[X_i,X_j])\\
 &=\varepsilon F(X_i,X_j).
\end{aligned}
\]
Also, $\nabla^\varepsilon_{\varepsilon^{-1}\xi}
(\varepsilon^{-1}\xi)=0$: its horizontal component is zero because
$g_\varepsilon(\varepsilon^{-1}\xi,X_i)=0$,
$[\xi,X_i]=0$, and $g_\varepsilon(\varepsilon^{-1}\xi,
\varepsilon^{-1}\xi)=1$, and its vertical component is again zero by
unit length.
To compute the Ricci trace, use the orthonormal frame
$X_1,X_2,\varepsilon^{-1}\xi$.  Then
\[
\begin{aligned}
 &\operatorname{Ric}_\varepsilon(\varepsilon^{-1}\xi,\varepsilon^{-1}\xi)\\
 &\quad=\sum_{i=1}^2g_\varepsilon\bigl(
 \nabla^\varepsilon_{X_i}
       \nabla^\varepsilon_{\varepsilon^{-1}\xi}(\varepsilon^{-1}\xi)
 -\nabla^\varepsilon_{\varepsilon^{-1}\xi}
       \nabla^\varepsilon_{X_i}(\varepsilon^{-1}\xi)
 -\nabla^\varepsilon_{[X_i,\varepsilon^{-1}\xi]}(\varepsilon^{-1}\xi),
 X_i\bigr)\\
 &\quad=-\sum_{i=1}^2g_\varepsilon\bigl(
 \nabla^\varepsilon_{\varepsilon^{-1}\xi}
       \nabla^\varepsilon_{X_i}(\varepsilon^{-1}\xi),X_i\bigr).
\end{aligned}
\]
The first and third terms vanish because
$\nabla^\varepsilon_{\varepsilon^{-1}\xi}(\varepsilon^{-1}\xi)=0$
and $[X_i,\varepsilon^{-1}\xi]=0$.
The preceding connection calculation says
\[
 \nabla^\varepsilon_{X_i}(\varepsilon^{-1}\xi)
 =\frac{\varepsilon}{2}\sum_{j=1}^2 F(X_i,X_j)X_j,
 \qquad
 \nabla^\varepsilon_{\varepsilon^{-1}\xi}X_j
 =\nabla^\varepsilon_{X_j}(\varepsilon^{-1}\xi).
\]
Since $(\varepsilon^{-1}\xi)(F(X_i,X_j))=0$, differentiating and
taking the inner product with $X_i$ yields
\[
\begin{aligned}
 g_\varepsilon\bigl(
 \nabla^\varepsilon_{\varepsilon^{-1}\xi}
 \nabla^\varepsilon_{X_i}(\varepsilon^{-1}\xi),X_i\bigr)
 &=\frac{\varepsilon}{2}\sum_{j=1}^2F(X_i,X_j)
 g_\varepsilon(\nabla^\varepsilon_{\varepsilon^{-1}\xi}X_j,X_i)\\
 &=\frac{\varepsilon^2}{4}\sum_{j=1}^2F(X_i,X_j)F(X_j,X_i).
\end{aligned}
\]
Substituting into the trace and using $F(X_j,X_i)=-F(X_i,X_j)$ gives
\[
\begin{aligned}
 \operatorname{Ric}_\varepsilon(\varepsilon^{-1}\xi,\varepsilon^{-1}\xi)
 &=-\frac{\varepsilon^2}{4}
   \sum_{i,j=1}^2F(X_i,X_j)F(X_j,X_i)\\
 &=\frac{\varepsilon^2}{4}\sum_{i,j=1}^2F(X_i,X_j)^2
 =\frac{\varepsilon^2}{2}|F|_h^2.
\end{aligned}
\tag{3.1}
\]
Here the two-form norm satisfies
$|F|_h^2=\frac12\sum_{i,j=1}^2F(X_i,X_j)^2$.

\emph{(ii) The horizontal component.} Let $X$ be a horizontal vector, extended from the base with
$\nabla^hX=0$ at the point of calculation.  The horizontal component of
$\nabla^\varepsilon_{X_i}X$ is $(\nabla^h_{X_i}X)^H$, and metric
compatibility with the vertical derivative computed above gives
\[
 \nabla^\varepsilon_{X_i}X
 =(\nabla^h_{X_i}X)^H-\frac12F(X_i,X)\xi.
\]
To compute the horizontal curvature trace, consider its three terms
separately.  Since $F(X,X)=0$, the first term is
\[
 g_\varepsilon(\nabla^\varepsilon_{X_i}\nabla^\varepsilon_X X,X_i)
 =h(\nabla^h_{X_i}\nabla^h_X X,X_i).
\]
For the second term, differentiate the preceding decomposition.
The derivative of the scalar coefficient multiplying $\xi$ has zero
inner product with $X_i$, so
\[
\begin{aligned}
 g_\varepsilon(\nabla^\varepsilon_X\nabla^\varepsilon_{X_i}X,X_i)
 &=h(\nabla^h_X\nabla^h_{X_i}X,X_i)
 -\frac12F(X_i,X)g_\varepsilon(\nabla^\varepsilon_X\xi,X_i)\\
 &=h(\nabla^h_X\nabla^h_{X_i}X,X_i)
 -\frac{\varepsilon^2}{4}F(X_i,X)F(X,X_i)\\
 &=h(\nabla^h_X\nabla^h_{X_i}X,X_i)
 +\frac{\varepsilon^2}{4}F(X_i,X)^2.
\end{aligned}
\]
Here we used
$g_\varepsilon(\nabla^\varepsilon_X\xi,X_i)
=\frac{\varepsilon^2}{2}F(X,X_i)$.
For the third term, the bracket decomposition gives
\[
\begin{aligned}
 g_\varepsilon(\nabla^\varepsilon_{[X_i,X]}X,X_i)
 &=h(\nabla^h_{[X_i,X]}X,X_i)
 -F(X_i,X)g_\varepsilon(\nabla^\varepsilon_\xi X,X_i)\\
 &=h(\nabla^h_{[X_i,X]}X,X_i)
 -\frac{\varepsilon^2}{2}F(X_i,X)F(X,X_i)\\
 &=h(\nabla^h_{[X_i,X]}X,X_i)
 +\frac{\varepsilon^2}{2}F(X_i,X)^2,
\end{aligned}
\]
where $\nabla^\varepsilon_\xi X=\nabla^\varepsilon_X\xi$.
The second and third terms are subtracted in the curvature expression,
so their positive corrections enter with negative signs:
\[
\begin{aligned}
 &g_\varepsilon\bigl(
 \nabla^\varepsilon_{X_i}\nabla^\varepsilon_X X
 -\nabla^\varepsilon_X\nabla^\varepsilon_{X_i}X
 -\nabla^\varepsilon_{[X_i,X]}X,X_i\bigr)\\
 &\quad =
 h\bigl(\nabla^h_{X_i}\nabla^h_X X
 -\nabla^h_X\nabla^h_{X_i}X
 -\nabla^h_{[X_i,X]}X,X_i\bigr)
 -\frac{\varepsilon^2}{4}F(X_i,X)^2
 -\frac{\varepsilon^2}{2}F(X_i,X)^2.
\end{aligned}
\]
The remaining, vertical direction in the Ricci trace contributes
$\frac{\varepsilon^2}{4}|\iota_XF|_h^2$, by the preceding vertical
curvature calculation with $X$ in place of $X_i$.
Consequently,
\[
 \operatorname{Ric}_\varepsilon(X,X)
 =\operatorname{Ric}_h(X,X)
 -\frac{3\varepsilon^2}{4}|\iota_XF|_h^2
 +\frac{\varepsilon^2}{4}|\iota_XF|_h^2.
\]
Polarization gives
\[
 \operatorname{Ric}_\varepsilon(X,Y)
 =\operatorname{Ric}_h(X,Y)
 -\frac{\varepsilon^2}{2}
 \langle\iota_XF,\iota_YF\rangle_h.
\]

\emph{(iii) The mixed component.} Only horizontal directions contribute
to the trace.  At the chosen point,
$\nabla^\varepsilon_{\varepsilon^{-1}\xi}X
=\frac{\varepsilon}{2}\sum_jF(X,X_j)X_j$.
The horizontal part of
$\nabla^\varepsilon_{\varepsilon^{-1}\xi}
\nabla^\varepsilon_{X_i}X$ vanishes there, since
$\nabla^h_{X_i}X=0$, all base coefficients are fiber-invariant, and
$\nabla^\varepsilon_{\varepsilon^{-1}\xi}\xi=0$.
Thus
\[
\begin{aligned}
 \operatorname{Ric}_\varepsilon(\varepsilon^{-1}\xi,X)
 &=\sum_{i=1}^2g_\varepsilon\bigl(
 \nabla^\varepsilon_{X_i}\nabla^\varepsilon_{\varepsilon^{-1}\xi}X
 -\nabla^\varepsilon_{\varepsilon^{-1}\xi}
       \nabla^\varepsilon_{X_i}X,X_i\bigr)\\
 &=\frac{\varepsilon}{2}
   \sum_{i=1}^2(\nabla^h_{X_i}F)(X,X_i)
 =\frac{\varepsilon}{2}(d_h^*F)(X),
\end{aligned}
\tag{3.2}
\]
where $(d_h^*F)(X)=-\sum_i(\nabla^h_{X_i}F)(X_i,X)$.
\end{proof}

\begin{proof}[Proof of Proposition~\ref{prop:uniform-adiabatic-estimate}]
By Lemma~\ref{lem:ricci-components}, the Ricci components are uniformly
bounded for $0<\varepsilon\le1$.  In particular, we have
\[
  \operatorname{Ric}_{g_\varepsilon}\ge-C_Rg_\varepsilon
  \qquad(0<\varepsilon\le1).
  \tag{3.3}
\]

The logarithmic amplitude cutoffs used for $\Ztwo$-harmonic forms allow the
Weitzenb\"ock formula to be integrated across $Z_\varepsilon$; see
\cite[Proposition~3.4]{HWZ}.  Since $I_\varepsilon$ is flat, this gives
\[
 \int_{Y\setminus Z_\varepsilon}|\nabla v_\varepsilon|^2
 d\operatorname{vol}_{g_\varepsilon}
 =-\int_{Y\setminus Z_\varepsilon}
 \operatorname{Ric}_{g_\varepsilon}(v_\varepsilon,v_\varepsilon)
 d\operatorname{vol}_{g_\varepsilon}\le C_R.
 \tag{3.4}
\]
Here the cutoff error is sent to zero before applying~\textup{(3.3)}, so no
constant associated with an individual zero set remains.  The normalization and
$d\operatorname{vol}_{g_\varepsilon}=\varepsilon d\mu$ imply
\[
 \int_Y|\sigma_\varepsilon|^2d\mu=1,
 \qquad
 \varepsilon\int_Y|\nabla^{g_\varepsilon}v_\varepsilon|^2d\mu\le C_R.
 \tag{3.5}
\]

The convergence of the rescaled connection to the split connection is a
standard fact in adiabatic geometry; see, for example,
\cite[Section~4, equation~(4.4)]{DaiWei}.  In our circle-fibered setting,
dualizing the covariant derivatives computed in the proof of
Lemma~\ref{lem:ricci-components}, for $\sigma=ue^0+\beta$ one obtains
\[
\begin{aligned}
 \left(\nabla^{\varepsilon,\mathrm{rescaled}}_X-\nabla^{\mathrm{split}}_X\right)\sigma
 &=\frac\varepsilon2\left(
 u\,\iota_XF-\langle\beta,\iota_XF\rangle e^0\right),\\
 \left(\nabla^{\varepsilon,\mathrm{rescaled}}_{\varepsilon^{-1}\xi}
 -\nabla^{\mathrm{split}}_{\varepsilon^{-1}\xi}\right)\sigma
 &=\frac\varepsilon2\iota_{\beta^{\sharp_h}}F.
\end{aligned}
 \tag{3.6}
\]
In particular,
\[
 \nabla^{\varepsilon,\mathrm{rescaled}}
 =\nabla^{\mathrm{split}}+O(\varepsilon),
\]
where the difference is measured using the fixed metric on $E$ and
$g_{\mathrm{ref}}$ on $Y$.
Tensoring with $I_\varepsilon$ does not alter these formulas, since the
connection on $I_\varepsilon$ contributes equally to both connections and
cancels in their difference.  The identity
$\Phi_\varepsilon\sigma_\varepsilon=\sqrt{\varepsilon}\,v_\varepsilon$
therefore gives, for a horizontal unit vector $X$,
\[
\begin{aligned}
 \nabla^{\mathrm{split}}_X\sigma_\varepsilon
 &=\sqrt{\varepsilon}\,\Phi_\varepsilon^{-1}
   \nabla^{g_\varepsilon}_{I_\varepsilon,X}v_\varepsilon\\
 &\quad-\frac{\varepsilon}{2}
 \left(u_\varepsilon\iota_XF
 -\langle\beta_\varepsilon,\iota_XF\rangle_h e^0\right).
\end{aligned}
\]
Since $\Phi_\varepsilon$ is a fiberwise isometry, the first term has norm
$\sqrt{\varepsilon}\,
|\nabla^{g_\varepsilon}_{I_\varepsilon,X}v_\varepsilon|$.
The horizontal and vertical summands in the second term are orthogonal,
and the boundedness of $F$ gives
\[
\begin{aligned}
 \left|u_\varepsilon\iota_XF
 -\langle\beta_\varepsilon,\iota_XF\rangle_h e^0\right|^2
 &=u_\varepsilon^2|\iota_XF|_h^2
   +|\langle\beta_\varepsilon,\iota_XF\rangle_h|^2\\
 &\le C^2\left(u_\varepsilon^2+|\beta_\varepsilon|_h^2\right)
 =C^2|\sigma_\varepsilon|^2.
\end{aligned}
\]
The triangle inequality now gives the horizontal estimate below.
Similarly, the second identity in~\textup{(3.6)} and
$|\iota_{\beta_\varepsilon^{\sharp_h}}F|_h
\le C|\beta_\varepsilon|_h\le C|\sigma_\varepsilon|$ give
\[
 |\nabla^{\mathrm{split}}_{\varepsilon^{-1}\xi}\sigma_\varepsilon|
 \le\sqrt{\varepsilon}\,
 |\nabla^{g_\varepsilon}_{I_\varepsilon,\varepsilon^{-1}\xi}v_\varepsilon|
 +C\varepsilon|\sigma_\varepsilon|.
\]
Multiplying by $\varepsilon$ and using
$\nabla^{\mathrm{split}}_\xi
=\varepsilon\nabla^{\mathrm{split}}_{\varepsilon^{-1}\xi}$
gives the vertical estimate.  Thus
\[
\begin{aligned}
 |\nabla^{\mathrm{split}}_X\sigma_\varepsilon|
 &\le \sqrt{\varepsilon}\,|\nabla^{g_\varepsilon}_{I_\varepsilon,X}
          v_\varepsilon|+C\varepsilon|\sigma_\varepsilon|,\\
 |\nabla^{\mathrm{split}}_\xi\sigma_\varepsilon|
 &\le \varepsilon^{3/2}
        |\nabla^{g_\varepsilon}_{I_\varepsilon,\varepsilon^{-1}\xi}
          v_\varepsilon|+C\varepsilon^2|\sigma_\varepsilon|.
\end{aligned}
 \tag{3.7}
\]
The combination of~\textup{(3.7)} and~\textup{(3.5)} proves both estimates
in~\textup{(2.12)}.  The calculation
lifts equivariantly to each of the finitely many cyclic Seifert charts, so a
single constant works on all of $Y$ and is independent of
$(Z_\varepsilon,I_\varepsilon)$.
\end{proof}

\subsection{The proof of Proposition~\ref{prop:adiabatic-limit}}

We need a lemma first.

\begin{lemma}\label{lem:uniform-local-regularity}
Let $g$ belong to a family of smooth Riemannian metrics on a fixed
three-dimensional domain $U$.  Assume that the family is
$C^\infty$-precompact and uniformly equivalent on any compact subset of $U$.
Let $(Z,I,v)$ be a $\Ztwo$-harmonic one-form on $(U,g)$ satisfying
prescribed upper bounds on
$\int_{K'}|v|_g^2\,d\operatorname{vol}_g$
for every relatively compact subdomain $K'\Subset U$.
Then, for every compact $K\Subset U$, there is a constant $C$, depending
only on $K$, the metric family, and these local $L^2$ bounds, but not on
the particular $g$, $Z$, $I$, or $v$, such that
\[
 |v(x)|_{g}\le C\,d_{g}(x,Z)^{1/2}\quad(x\in K),
 \qquad
 \bigl\|\,|v|_{g}\,\bigr\|_{C^{0,1/3}(K)}\le C.
\]
Here $d_g(x,Z)$ denotes the distance from $x$ to $Z$ with respect to $g$,
with $d_g(x,\varnothing)=\infty$.
The H\"older norm may be measured in any fixed smooth background metric.
\end{lemma}

\begin{remark}
These estimates are a minor adaptation of Taubes's local regularity
argument~\cite[Theorem~1.5, Sections~2--3, and Lemma~6.4]{Taubes}.
His statements do not explicitly address uniformity for the varying
metrics and zero sets considered here, so we give the argument with this
dependence made explicit.
\end{remark}

\begin{proof}
Choose $K'\Subset U$ containing $K$ in its interior and a common radius
$R>0$ so that the balls used below lie in $K'$ and in normal coordinate
charts for every metric in the family.  All geometric constants below
are uniform over this family and in the center of the ball; constants
in bounds for $v$ may also depend on the prescribed $L^2$ bound on $K'$.

Fix $g$ and $(Z,I,v)$, and use $g$ throughout the calculation.  With $\Delta=\operatorname{div}\nabla$, the scalar
Weitzenb\"ock identity gives
\[
 \frac12\Delta|v|^2
 =|\nabla v|^2+\operatorname{Ric}(v,v)\ge -C|v|^2.
\]
This holds weakly across the zero set.  Indeed, cutoffs that vanish where
$|v|\le\delta$ and equal one where $|v|\ge2\delta$ introduce errors bounded
by $C\int_{\{0<|v|<2\delta\}}|\nabla v|^2$, which tend to zero.

To obtain the local supremum estimate, choose a cutoff $0\le\chi\le1$
supported in $B_t(x)$, equal to one on $B_s(x)$, with
$0<s<t\le R$ and $|\nabla\chi|\le2/(t-s)$.
The Weitzenb\"ock identity and $|d|v||\le|\nabla v|$ give
\[
 \Delta\sqrt{|v|^2+\delta^2}\ge-C\sqrt{|v|^2+\delta^2}.
\]
For $m\ge2$, multiply this inequality by
$\chi^2(|v|^2+\delta^2)^{(m-1)/2}$ and integrate by parts.
Absorbing the term containing $\nabla\chi$ by Young's inequality and
letting $\delta\downarrow0$ gives
\[
 \int_{B_t(x)}\bigl|\nabla(\chi|v|^{m/2})\bigr|^2
 \le Cm^2\int_{B_t(x)}(|\nabla\chi|^2+1)|v|^m.
\]
The three-dimensional Sobolev inequality then yields
\[
 \left(\int_{B_s(x)}|v|^{3m}\right)^{1/3}
 \le Cm^2\bigl((t-s)^{-2}+1\bigr)\int_{B_t(x)}|v|^m.
\]
Moser iteration with exponents $m=2,6,18,\ldots$ on concentric balls whose
radii decrease from $t$ to $t/2$ gives
\[
 |v(x)|^2\le Ct^{-3}\int_{B_t(x)}|v|^2.
\]
In particular, the local $L^2$ bound gives a uniform supremum bound on
a fixed neighborhood of $K$.

The conclusion of Lemma~\ref{lem:uniform-local-regularity} is immediate on
components where $v$ vanishes identically.  Otherwise, fix $p\in Z$ near $K$ and set
\[
 H(r)=\int_{\partial B_r(p)}|v|^2,\qquad
 E(r)=\int_{B_r(p)}|\nabla v|^2,\qquad
 n(r)=\frac{rE(r)}{H(r)}.
\]
To compare volume and boundary integrals, test
$\frac12\Delta|v|^2\ge-C|v|^2$ against $(r^2-d(p,\cdot)^2)/6$.
This function is zero on the boundary,
has outward normal derivative $-r/3$, and has Laplacian $-1+O(r^2)$.
Integration by parts therefore gives
\[
 (1-Cr^2)\int_{B_r(p)}|v|^2\le\frac r3H(r).
\]
After uniformly decreasing $R$, this implies
$\int_{B_r(p)}|v|^2\le CrH(r)$.
Since the zero set of a nonzero form has empty interior
by~\cite[Lemma~3.1 and~(3.1)]{Taubes}, this bound also implies
$H(r)>0$ for $0<r\le R$.
Integrating the Weitzenb\"ock identity
\[
 \frac12\Delta|v|^2=|\nabla v|^2+\operatorname{Ric}(v,v)
\]
over $B_r(p)$ gives, for almost every $r$,
\[
 \int_{\partial B_r(p)}\langle v,\nabla_{\partial r}v\rangle
 =E(r)+\int_{B_r(p)}\operatorname{Ric}(v,v)
 =E(r)+O(rH(r)).
\]
In geodesic polar coordinates, identify $\partial B_r(p)$ with the unit
sphere in $T_pU$ by $\theta\mapsto\exp_p(r\theta)$.  Its area element
satisfies
\[
 dS_g=r^2(1+O(r^2))\,d\theta,
 \qquad
 \partial_r dS_g=\left(\frac2r+O(r)\right)dS_g.
\]
Differentiating $H(r)=\int_{\partial B_r(p)}|v|^2$ therefore gives,
for almost every $r$,
\[
\begin{aligned}
 H'(r)
 &=2\int_{\partial B_r(p)}\langle v,\nabla_{\partial r}v\rangle
       +\frac2rH(r)+O(rH(r))\\
 &=\frac2rH(r)+2E(r)+O(rH(r)).
\end{aligned}
\]

Next, we estimate the derivative of $E(r)$.  Let $X$ be a smooth vector
field compactly supported in $K'$, with flow $\phi_t$.  On
$K'\setminus\phi_t^{-1}(Z)$, give $I_t=\phi_t^*I$ its pulled-back flat
connection and define $v_t(x)$ by parallel transporting the covector
factor of $v(\phi_t(x))$ back to $x$ along the flow line.  Then
$|v_t|=|v|\circ\phi_t$, and $v_t,\nabla v_t$ are in $L^2(K')$.
After pulling the operators back to the fixed domain, their difference
from the original operators has first- and zeroth-order coefficients
of size $O(t)$.  Since $d_Iv=d_I^*v=0$, this gives
\[
 \|d_{I_t}v_t\|_{L^2(K')}+\|d_{I_t}^*v_t\|_{L^2(K')}
 \le C_X|t|\bigl(\|\nabla v\|_{L^2(K')}+\|v\|_{L^2(K')}\bigr).
\]

Choose a fixed $0\le\psi\le1$ in $C_c^\infty(K')$, equal to one near
$\operatorname{supp}X$, to remove boundary terms at $\partial K'$.
Before moving to the next step, we must justify the following integrated
Weitzenb\"ock identity across the moved zero set:
\[
 \int_{K'}\bigl(|d_{I_t}(\psi v_t)|^2
                  +|d_{I_t}^*(\psi v_t)|^2\bigr)
 =\int_{K'}\bigl(|\nabla(\psi v_t)|^2
                  +\operatorname{Ric}(\psi v_t,\psi v_t)\bigr).
\]
For each fixed $t$, cut off the moved zero set using a smooth function
$\chi_\delta$ that is zero where $|v_t|\le\delta^2$ and one where
$|v_t|\ge\delta$, with $0<\delta<1$.
Choosing the transition as a function of $\log|v_t|$, Kato's inequality
gives
\[
 \|v_t\,d\chi_\delta\|_{L^2(K')}
 \le \frac{C}{|\log\delta|}\|\nabla v_t\|_{L^2(K')}
 \longrightarrow0\qquad(\delta\downarrow0).
\]
The terms containing $\chi_\delta-1$ tend to zero by dominated
convergence, and this estimate controls the derivative of the cutoff.
Hence $\psi\chi_\delta v_t\to\psi v_t$ strongly in $W^{1,2}$.
Each approximating form has compact support away from the moved zero
set, so the usual integrated Weitzenb\"ock identity applies to it.
Passing to the limit proves this identity.

We now compute the change in the Hodge energy
\[
 \int_{K'}\bigl(|d_{I_t}(\psi v_t)|^2
                  +|d_{I_t}^*(\psi v_t)|^2\bigr).
\]
The choice $\psi=1$ near $\operatorname{supp}X$ ensures that the form
and connection are unchanged on $\operatorname{supp}d\psi$.
Thus the terms containing $d\psi$ in the Leibniz rule contribute the
same energy for every $t$; their cross terms with $d_{I_t}v_t$ and
$d_{I_t}^*v_t$ vanish because both residuals are zero there.
Consequently, the preceding residual estimate gives
\[
\begin{aligned}
 &\int_{K'}\bigl(|d_{I_t}(\psi v_t)|^2
                    +|d_{I_t}^*(\psi v_t)|^2\bigr)\\
 &\quad-\int_{K'}\bigl(|d_I(\psi v)|^2+|d_I^*(\psi v)|^2\bigr)\\
 &=\int_{K'}\psi^2\bigl(|d_{I_t}v_t|^2+|d_{I_t}^*v_t|^2\bigr)
 =O(t^2).
\end{aligned}
\]
The integrated identity therefore implies that the first variation
vanishes:
\[
 \left.\frac{d}{dt}\right|_{t=0}
 \int_{K'}\bigl(|\nabla(\psi v_t)|^2
       +\operatorname{Ric}(\psi v_t,\psi v_t)\bigr)=0.
\]

We expand this first variation as follows:
\[
\begin{aligned}
 &\left.\frac{d}{dt}\right|_{t=0}
   \int_{K'}\bigl(|\nabla(\psi v_t)|^2
                 +\operatorname{Ric}(\psi v_t,\psi v_t)\bigr)\\
 &=\int_{K'}\Bigl(2\sum_{i=1}^3
       \langle\nabla_{e_i}v,\nabla_{\nabla_{e_i}X}v\rangle
       -(\operatorname{div}X)|\nabla v|^2\\
 &\hspace{55pt}-2\sum_{i=1}^3
       \langle\nabla_{e_i}v,\operatorname{Rm}(X,e_i)v\rangle
       -(\nabla_X\operatorname{Ric})(v,v)\\
 &\hspace{55pt}-(\operatorname{div}X)\operatorname{Ric}(v,v)\Bigr)=0.
\end{aligned}
\]
Here $e_1,e_2,e_3$ form a local orthonormal frame, and
$\operatorname{Rm}(X,Y)=[\nabla_X,\nabla_Y]-\nabla_{[X,Y]}$
is the curvature acting on covectors; the sign line contributes no
curvature because it is flat.

The cutoff $\psi$ contributes no variation, since $\psi=1$ near
$\operatorname{supp}X$ and the flow is stationary elsewhere.
Hence
\[
\begin{aligned}
 &\left|\int_{K'}\left((\operatorname{div}X)|\nabla v|^2
     -2\sum_{i=1}^3
       \langle\nabla_{e_i}v,\nabla_{\nabla_{e_i}X}v\rangle\right)\right|\\
 &\qquad\le C\int_{K'}\left((|\nabla X|+|X|)|v|^2
                        +|X|\,|v|\,|\nabla v|\right).
\end{aligned}
\]
Here $C$ uses only bounds for the curvature and its first covariant
derivative, and is independent of $Z$, $I$, and $v$.

Write $\rho=d_g(p,\cdot)$ and take
$X=\chi(\rho)\rho\partial_\rho$, where $\chi$ is a smooth nonincreasing
radial cutoff equal to one for $\rho\le r$ and zero for $\rho\ge r+\delta$,
with $0<r<r+\delta<R$ and $|\chi'|\le C/\delta$.

Since $\nabla(\rho\partial_\rho)=\operatorname{id}+O(\rho^2)$,
we have
\[
\begin{aligned}
 &\int_{K'}\left((\operatorname{div}X)|\nabla v|^2
     -2\sum_{i=1}^3
       \langle\nabla_{e_i}v,\nabla_{\nabla_{e_i}X}v\rangle\right)\\
 &=\int_{K'}\left(\chi|\nabla v|^2
     +\rho\chi'\bigl(|\nabla v|^2-2|\nabla_{\partial_\rho}v|^2\bigr)\right)\\
 &\qquad+O\!\left(\int_{K'}\rho^2|\chi|\,|\nabla v|^2\right).
\end{aligned}
\]
Let $\delta\downarrow0$, so the transition annulus shrinks to
$\partial B_r(p)$.  Since $|\nabla\rho|=1$ away from $p$, the coarea
formula gives, for almost every $r\in(0,R)$,
\[
\begin{aligned}
 \int_{K'}\chi|\nabla v|^2
 &=\int_0^{r+\delta}\chi(s)
      \left(\int_{\partial B_s(p)}|\nabla v|^2\right)ds
   \longrightarrow E(r),\\
 \int_{K'}\rho\chi'|\nabla v|^2
 &=\int_r^{r+\delta}s\chi'(s)
      \left(\int_{\partial B_s(p)}|\nabla v|^2\right)ds\\
 &\longrightarrow-r\int_{\partial B_r(p)}|\nabla v|^2=-rE'(r),\\
 -2\int_{K'}\rho\chi'|\nabla_{\partial_\rho}v|^2
 &=-2\int_r^{r+\delta}s\chi'(s)
      \left(\int_{\partial B_s(p)}|\nabla_{\partial_s}v|^2\right)ds\\
 &\longrightarrow2r\int_{\partial B_r(p)}|\nabla_{\partial r}v|^2.
\end{aligned}
\]
Here $\int_r^{r+\delta}\chi'(s)\,ds=-1$; the last two limits hold at
Lebesgue points of the corresponding sphere integrals, using
$|\chi'|\le C/\delta$.
The geometric error is bounded in the limit by $Cr^2E(r)$.
Combining these limits with the preceding curvature estimate gives
\[
\begin{aligned}
 rE'(r)-2r\int_{\partial B_r(p)}|\nabla_{\partial r}v|^2
 ={}&E(r)+O(r^2E(r))\\
 &+O\!\left(r\int_{B_r(p)}|v|\,|\nabla v|
             +\int_{B_r(p)}|v|^2+rH(r)\right).
\end{aligned}
\]
Using $|v|\,|\nabla v|\le r|\nabla v|^2+r^{-1}|v|^2$ and the volume
bound yields
\[
 E'(r)\ge2\int_{\partial B_r(p)}|\nabla_{\partial r}v|^2
          +\left(\frac1r-Cr\right)E(r)-CH(r).
\]
The coarea formula and $|v|^2\in W^{1,1}$ on compact annuli give local
absolute continuity of $E$ and $H$ on $(0,R)$.  Differentiating $n$,
substituting the preceding estimates, and applying Cauchy--Schwarz gives
\[
\begin{aligned}
 n'(r)
 &\ge\frac{2r}{H(r)}
       \int_{\partial B_r(p)}|\nabla_{\partial r}v|^2
       -\frac{2n(r)^2}{r}-Cr(1+n(r))\\
 &\ge2r\left(\frac{E(r)+O(rH(r))}{H(r)}\right)^2
       -\frac{2n(r)^2}{r}-Cr(1+n(r))\\
 &\ge-Cr(1+n(r)).
\end{aligned}
\]
The constant is enlarged in the last line.

Viewing $v$ as a degree-one solution of the Hodge--Dirac equation,
the frequency in~\cite[(3.2)--(3.3)]{Taubes} is
$(1+O(r^2))n(r)$ in dimension three; the factor comes from his
exponential boundary normalization.
Thus~\cite[Lemma~6.4]{Taubes} gives $n(0)\ge1/2$ at $p\in Z$ for
each fixed form.  Integration gives
\[
 n(r)+1\ge e^{-Cr^2/2}(n(0)+1),\qquad
 n(r)\ge\frac12-Cr^2.
\]
Consequently,
\[
 \frac{d}{dr}\log\frac{H(r)}{r^2}\ge\frac1r-Cr,
 \qquad
 H(r)\le C H(R)\left(\frac rR\right)^3.
\]
The uniform supremum bound controls $H(R)$ at the common radius $R$.
Hence $\int_{B_r(p)}|v|^2=\int_0^rH(t)\,dt\le Cr^4$.
For $d(p,x)=\rho<R/4$, the local supremum estimate gives
\[
 |v(x)|^2\le C\rho^{-3}\int_{B_\rho(x)}|v|^2
 \le C\rho^{-3}\int_{B_{2\rho}(p)}|v|^2\le C\rho.
\]
Taking the distance to $Z$, and using the supremum bound for larger
distances, proves $|v(x)|\le C d(x,Z)^{1/2}$ on $K$.

Finally, consider $x,y\in K$ with $d(x,y)=\rho\le1$ sufficiently small.
If $d(x,Z)\le2\rho^{2/3}$, then
\[
 d(y,Z)\le\rho+d(x,Z)\le3\rho^{2/3}.
\]
The zero-distance bound therefore gives
$\bigl||v(x)|-|v(y)|\bigr|\le C\rho^{1/3}$; the same applies with
$x,y$ exchanged.  Otherwise, a minimizing geodesic joining them stays
at distance at least $\rho^{2/3}$ from $Z$.  Balls of radius
$\rho^{2/3}/2$ centered along it avoid $Z$.  In a parallel unit frame
of $I$, the form $v$ is an ordinary harmonic one-form on each such ball.
Rescaling the ball to unit radius and applying the interior gradient
estimate gives, at each point $q$ of the geodesic,
\[
 |\nabla v(q)|\le C\rho^{-2/3}
      \sup_{B_{\rho^{2/3}/2}(q)}|v|
 \le C\rho^{-2/3},
\]
where the last inequality uses the uniform supremum bound.  Thus
\[
 \bigl||v(x)|-|v(y)|\bigr|
 \le\int_0^\rho C\rho^{-2/3}\,dt=C\rho^{1/3}.
\]
Larger distances are again covered by the supremum bound.  Uniform
equivalence with a fixed background metric gives the stated H\"older norm.
\end{proof}

\begin{proof}[Proof of Proposition~\ref{prop:adiabatic-limit}]
We first remove the difficulty caused by the varying sign lines.  On each
fiber of $E/\{\pm1\}$ define
\[
 d([a],[b])=\min\{|a-b|,|a+b|\},
 \qquad
 \iota([a])=
 \begin{cases}
   a\otimes a/|a|,&a\ne0,\\
   0,&a=0.
 \end{cases}
 \tag{3.8}
\]
If $a=r\widehat a$, $b=s\widehat b$, and
$c=|\langle\widehat a,\widehat b\rangle|$, then
\[
 d([a],[b])^2=r^2+s^2-2rsc,
 \qquad
 |\iota([a])-\iota([b])|^2=r^2+s^2-2rsc^2.
\]
Consequently
\[
 d([a],[b])
 \le |\iota([a])-\iota([b])|
 \le\sqrt2\,d([a],[b]).
 \tag{3.9}
\]
For a local branch $a=r\widehat a$ one likewise has
\[
 |\nabla^{\mathrm{split}}a|
 \le|\nabla^{\mathrm{split}}\iota([a])|
 \le\sqrt2\,|\nabla^{\mathrm{split}}a|.
 \tag{3.10}
\]
The logarithmic cutoff at $|\sigma_j|=0$ makes~\textup{(3.10)} valid globally in
the weak sense; this is the standard two-valued Sobolev extension argument
of~\cite[Lemmas~2.1 and~2.3]{Zhang}.  Thus
$s_j:=\iota(\sigma_j)$ is uniformly bounded in
$W^{1,2}(Y;\operatorname{Sym}^2E)$ by Proposition~\ref{prop:uniform-adiabatic-estimate}.
Rellich compactness gives, after passing to a subsequence,
\[
 s_j\rightharpoonup s_\infty\quad\hbox{in }W^{1,2},
 \qquad
 s_j\longrightarrow s_\infty\quad\hbox{in }L^2.
 \tag{3.11}
\]
Here $\rightharpoonup$ denotes weak convergence and $\longrightarrow$ denotes strong convergence.
The image of $\iota$ is the closed cone of nonnegative rank-at-most-one
symmetric tensors.  Hence $\sigma_\infty:=\iota^{-1}(s_\infty)$ is defined almost
everywhere, and~\textup{(3.9)} gives
\[
 \int_Yd(\sigma_j,\sigma_\infty)^2d\mu\longrightarrow0,
 \qquad
 \int_Y|\sigma_\infty|^2d\mu=1.
 \tag{3.12}
\]
Moreover,~\textup{(2.12)} and~\textup{(3.10)} give
$\|\nabla^{\mathrm{split}}_\xi s_j\|_{L^2}\le C\varepsilon_j$.
Passing to the weak limit shows $\nabla^{\mathrm{split}}_\xi s_\infty=0$.
Since $\iota$ is injective and equivariant, $\sigma_\infty$ is
$S^1$-invariant.

We next identify its equation.  On a regular Seifert chart write
$\eta=dt+A$, where $t\in\mathbb R/(2\pi\mathbb Z)$, and unwrap the fibers by
\[
 F_j:D_0\times\mathbb R
 \longrightarrow D_0\times\bigl(\mathbb R/(2\pi\mathbb Z)\bigr),
 \qquad F_j(x,s)=\bigl(x,[s/\varepsilon_j]_{2\pi}\bigr).
\]
For each fixed $x$, this is the universal covering of the circle
fiber.
Let $S>0$ be fixed independently of $j$.  On the slab $D_0\times[-S,S]$,
\[
 F_j^*g_{\varepsilon_j}
 =h+(ds+\varepsilon_jA)^2\longrightarrow h+ds^2
 \quad\hbox{in }C^\infty.
 \tag{3.13}
\]
Writing $\sigma_j=u_je^0+\beta_j$, the normalized pullbacks are
\[
 \sqrt{\varepsilon_j}\,F_j^*v_j
 =u_j(x,s/\varepsilon_j)(ds+\varepsilon_jA)
  +\beta_j(x,s/\varepsilon_j).
 \tag{3.14}
\]
For a nonnegative $2\pi$-periodic function $f$, period counting gives
\[
 \varepsilon_j\int_{-S/\varepsilon_j}^{S/\varepsilon_j}f(t)dt
 \le\left(\frac S\pi+2\varepsilon_j\right)
 \int_0^{2\pi}f(t)dt.
 \tag{3.15}
\]
Since $A$ is bounded on compact sets and the pulled-back coefficients
have uniformly bounded local $L^2$ norm by~\textup{(3.15)}, the term
$\varepsilon_j u_j(x,s/\varepsilon_j)A$ in~\textup{(3.14)} tends to zero
in $L^2_{\mathrm{loc}}$.
Thus, using~\textup{(3.12)},~\textup{(3.15)}, and the fiber-invariance of
$\sigma_\infty$, and replacing $e^0$ by $ds$ in $\sigma_\infty$, we obtain
\[
 \sqrt{\varepsilon_j}\,F_j^*v_j\longrightarrow w_\infty
 \quad\hbox{strongly in }L^2_{\mathrm{loc}},
 \tag{3.16}
\]
as quotient-valued one-forms.  On the nonzero locus, a local representative
$\sigma_\infty=u_0e^0+\beta_0$ corresponds to
$w_\infty=u_0\,ds+\beta_0$, with coefficients independent of $s$.

Equations~\textup{(2.12)} and~\textup{(3.15)} also give a uniform
$W^{1,2}(D_0\times[-S,S])$ bound for
$\sqrt{\varepsilon_j}\,F_j^*v_j$, measured with $h+ds^2$.
For nested compact slabs $K\Subset K'$, the metrics
in~\textup{(3.13)} and the preceding local $L^2$ bound satisfy the
hypotheses of Lemma~\ref{lem:uniform-local-regularity}.
Applying it to $\sqrt{\varepsilon_j}F_j^*v_j$ gives
\[
 \bigl\|\,|\sqrt{\varepsilon_j}F_j^*v_j|\,\bigr\|_{C^{0,1/3}(K)}
 \le C_{K,K'},
 \tag{3.17}
\]
and
\[
 |(\sqrt{\varepsilon_j}F_j^*v_j)(x)|
 \le C\,d_{F_j^*g_{\varepsilon_j}}(x,F_j^{-1}(Z_j))^{1/2}.
 \tag{3.18}
\]
Near an exceptional fiber, first pull back the metric and the one-form
to the smooth finite cyclic cover $D_0\times S^1$ of the Seifert chart,
and then unwrap the circle factor as above.  On this cover the metric coefficients
are smooth across the center of $D_0$, so the same estimates apply.
The cyclic group preserves the metric and the norm of the pulled-back
form, and hence the norm estimates descend to the local quotient.
The covering degree is fixed independently of $j$, so this does not
affect uniformity in $j$.

By Arzel\`a--Ascoli and~\textup{(3.16)}, the scalar norms converge locally
uniformly to $|w_\infty|$.  On a simply connected ball disjoint from
$Z_\infty:=\{|w_\infty|=0\}$, the approximating forms are eventually
nonvanishing.  After choosing one sign, ordinary elliptic estimates for
$(d+d^*_{F_j^*g_{\varepsilon_j}})(\sqrt{\varepsilon_j}\,F_j^*v_j)=0$ give smooth convergence to a
local branch of $w_\infty$ satisfying
\[
 dw_\infty=0,
 \qquad d_{h+ds^2}^*w_\infty=0.
 \tag{3.19}
\]
Estimate~\textup{(3.17)} passes to the limit and implies, for
$z\in Z_\infty$,
\[
 \int_{B_\rho(z)}|w_\infty|^2
 d\operatorname{vol}_{h+ds^2}\le C\rho^{3+2/3}.
 \tag{3.20}
\]
The logarithmic cutoff therefore extends~\textup{(3.19)} distributionally
across the zero set.  Thus $w_\infty$ is a genuine $\Ztwo$-harmonic one-form
for the product metric.

We use the following three-dimensional form of Zhang's zero-set
theorem~\cite[Theorem~1.4 and the subsequent three-dimensional consequence]{Zhang};
the one-form formulation is also recorded in
\cite[Definition~3.1 and Theorem~3.3]{HWZ}.
Let $(M,g)$ be a connected smooth Riemannian three-manifold without
boundary, not necessarily compact or complete, and let $v$ be a nonzero
continuous section of $T^*M/\{\pm1\}$ with zero set $Z$.
Suppose that its local branches on $M\setminus Z$ are smooth and satisfy
$dv=d_g^*v=0$, that
\[
 \int_{M\setminus Z}(|v|^2+|\nabla v|^2)\,d\operatorname{vol}_g<\infty,
\]
and that, for some $\delta>0$, every $z\in Z$ has constants
$C_z,\rho_z>0$ such that
\[
 \int_{B_\rho(z)}|v|^2\,d\operatorname{vol}_g
 \le C_z\rho^{3+\delta}\qquad(0<\rho<\rho_z).
\]
Here the balls and norms are taken with respect to $g$.
The conclusion is that $Z$ is one-rectifiable and
$\mathcal H_g^1(Z\cap K)<\infty$ for every compact subset $K\subset M$.

On each connected relatively compact product subchart where the limit
is nonzero, its continuous norm, finite energy, and~\textup{(3.20)}
allow us to apply Taubes's empty-interior result
\cite[Lemma~3.1 and~(3.1)]{Taubes}.
Normalization~\textup{(3.12)} supplies one such nonzero subchart;
overlapping connected subcharts then propagate nontriviality, so
$w_\infty$ cannot vanish identically on any nonempty open subset of a
product chart.

Now restrict $w_\infty$ to a connected open subset
$\Omega\Subset D_0\times\mathbb R$ of a smooth product chart.
The local $W^{1,2}$ bound gives finite $L^2$ norm and Dirichlet energy on
$\Omega$.
Taking $(M,g)=(\Omega,h+ds^2)$ in the stated result shows that
$Z_\infty\cap\Omega$ is one-rectifiable and
\[
 \mathcal H^1_{h+ds^2}(Z_\infty\cap K)<\infty
 \qquad\hbox{for every compact }K\subset\Omega.
\]
On compact subcharts this metric is uniformly equivalent to the lift of
the fixed metric $g_{\mathrm{ref}}=g_1$.  The local Seifert covering maps
and a finite covering of $Y$ consequently give
\[
 \mathcal H^1_{g_{\mathrm{ref}}}\bigl(Z(\sigma_\infty)\bigr)<\infty.
\]
Since $\sigma_\infty$ is $S^1$-invariant, its zero set is a union of circle
fibers.  In $g_{\mathrm{ref}}$, the generator $\xi$ has unit length;
regular fibers have length $2\pi$ and the exceptional fibers have lengths
$2\pi/\alpha_i$.  Finite total length implies that only
finitely many fibers occur.  Therefore
\[
 Z(\sigma_\infty)=\pi^{-1}(\Sigma_\infty)
 \tag{3.21}
\]
for a finite subset $\Sigma_\infty\subset B$.

We first check compatibility with the cyclic action near an exceptional
fiber.  In a cyclic uniformizer of order $\alpha$, let
$\zeta=e^{2\pi i/\alpha}$ and let the generator act by
$(z,t)\mapsto(\zeta z,t+2\pi b/\alpha)$.  Since $t=s/\varepsilon_j$,
its lift to $D_0\times\mathbb R$ is
\[
 \gamma_j(z,s)
   =(\zeta z,s+2\pi b\varepsilon_j/\alpha)
 \longrightarrow
 \gamma_0(z,s)=(\zeta z,s).
 \tag{3.21a}
\]
The pulled-back quotient forms are invariant under $\gamma_j$, since they
come from the Seifert neighborhood.  Strong local convergence then gives
$\gamma_0^*w_\infty=w_\infty$ as quotient-valued one-forms.
Thus, on the nonzero locus, the local branches and their sign line are
compatible with the cyclic action on the uniformizing disk.

On the nonzero locus, transporting a representative of $\sigma_\infty$
by the circle action closes with the same vector after time $2\pi$,
because the time-$2\pi$ action on $E$ is the identity.  Thus the
regular-fiber holonomy of the limiting flat line bundle is $+1$.
The orbifold homotopy exact sequence for the Seifert fibration shows that
its holonomy representation factors through
$\pi_1^{\orb}(B\setminus\Sigma_\infty)$.
Hence the line bundle descends to an orbifold flat line bundle on
$B\setminus\Sigma_\infty$.

Finally, since the coefficients of $w_\infty=u_0ds+\beta_0$ are
independent of $s$,
\[
 d(u_0ds+\beta_0)
 =d_B\beta_0+d_Bu_0\wedge ds,
 \qquad
 d^*(u_0ds+\beta_0)=d_B^*\beta_0,
\]
Equation~\textup{(3.19)} therefore gives $d_Bu_0=0$ and
$d_B\beta_0=d_B^*\beta_0=0$, as asserted in~\textup{(2.13)}.

If $u_0\equiv0$, then on the complement of $\Sigma_\infty$ each local
$(1,0)$-part $\beta_0^{1,0}$ is holomorphic with values in the sign line, so
its square is single-valued and holomorphic.  Estimate~\textup{(3.17)} gives,
in every disk or cone uniformizer,
\[
 |\beta_0(x)|\le C d_B(x,\Sigma_\infty)^{1/3},
 \qquad
 \int_{D_\rho}|\beta_0|^2dA_h\le C\rho^{2+2/3}.
 \tag{3.22}
\]
The resulting quadratic differential has a removable singularity at each
point of $\Sigma_\infty$.  Equivariance on the cone uniformizers shows that
the extension is an orbifold quadratic differential.  This proves every
assertion of Proposition~\ref{prop:adiabatic-limit}.
\end{proof}

\subsection{The proof of Proposition~\ref{prop:zero-exclusion-energy}}

\begin{proof}[Proof of Proposition~\ref{prop:zero-exclusion-energy}]
Suppose first that $Z_j\ne\varnothing$ along an infinite subsequence and
choose $p_j\in Z_j$.  A finite collection of nested Seifert charts
covers $Y$, so, after passing to a subsequence,
all $p_j$ lift to one such chart.  In a regular chart choose a full-fiber deck
translate so that $|s_j|\le\pi\varepsilon_j$; at an exceptional fiber first
lift to the fixed cyclic uniformizer and make the same choice upstairs.  The
lifted points therefore remain in one fixed compact slab.  The local uniform
convergence in Proposition~\ref{prop:adiabatic-limit} and the assumption
$\sigma_\infty=\{\pm ce^0\}$ give
\[
 |(\sqrt{\varepsilon_j}\,F_j^*v_j)|\longrightarrow c>0
 \quad\hbox{uniformly on that slab}.
 \tag{3.23}
\]
This contradicts the fact that the lift of $p_j$ is a zero.  Hence
$Z_j=\varnothing$ for all sufficiently large $j$.

The exact fiberwise identity
\[
 d(\sigma_j,\{\pm ce^0\})^2
 =(|u_j|-c)^2+|\beta_j|^2
 \tag{3.24}
\]
and strong convergence give
\[
 \|\beta_j\|_{L^2(d\mu)}\longrightarrow0,
 \qquad
 \bigl\|\,|u_j|-c\,\bigr\|_{L^2(d\mu)}\longrightarrow0.
 \tag{3.25}
\]

It remains to prove the derivative decay.  Set
\[
 \mathcal E_j
 =\varepsilon_j\int_Y|\nabla^{g_{\varepsilon_j}}_{I_j}v_j|^2d\mu.
\]
Scaling the cutoff Weitzenb\"ock identity~\textup{(3.4)} gives
\[
 \mathcal E_j
 =-\varepsilon_j\int_Y\operatorname{Ric}_{g_{\varepsilon_j}}
   (v_j,v_j)d\mu.
 \tag{3.26}
\]
The metric dual of $\sqrt{\varepsilon_j}\,v_j$ is
$\varepsilon_j^{-1}u_j\xi+\beta_j^{\sharp_h}$.  Using the three Ricci identities
in Lemma~\ref{lem:ricci-components}, we obtain
\[
\begin{aligned}
 \varepsilon_j\operatorname{Ric}_{\varepsilon_j}(v_j,v_j)
 ={}&\operatorname{Ric}_h(\beta_j^{\sharp_h},\beta_j^{\sharp_h})
 -\frac{\varepsilon_j^2}{2}
 |\iota_{\beta_j^{\sharp_h}}F|_h^2\\
 &+\varepsilon_j u_j(d_h^*F)(\beta_j^{\sharp_h})
 +\frac{\varepsilon_j^2}{2}|F|_h^2u_j^2.
\end{aligned}
 \tag{3.27}
\]
Since all base tensors are fixed,~\textup{(3.25)} and
$\|u_j\|_{L^2(d\mu)}\le\|\sigma_j\|_{L^2(d\mu)}=1$ give
\[
 0\le\mathcal E_j
 \le C\|\beta_j\|_{L^2(d\mu)}^2
   +C\varepsilon_j\|u_j\|_{L^2(d\mu)}\|\beta_j\|_{L^2(d\mu)}
   +C\varepsilon_j^2\|u_j\|_{L^2(d\mu)}^2
 \longrightarrow0.
 \tag{3.28}
\]
Keeping the small factors in the connection comparison~\textup{(3.7)} gives
\[
 \int_Y|\nabla^{\mathrm{split}}\sigma_j|_{g_{\mathrm{ref}}}^2d\mu
 \le C\mathcal E_j+C\varepsilon_j^2\longrightarrow0.
 \tag{3.29}
\]
Together with~\textup{(3.25)}, this is precisely~\textup{(2.17)} and proves
the proposition.
\end{proof}

\subsection{The proof of Proposition~\ref{prop:nonorientable-dichotomy}(ii)}
\label{sec:nonorientable-nonexistence}

\begin{proof}[Proof of Proposition~\ref{prop:nonorientable-dichotomy}\textup{(ii)}]
Let the underlying nonorientable surface of the base have genus $k$.
The presentation in~\cite[Theorem~6.1]{JN}, with the integer Seifert
invariant written separately, has generators
$a_1,\ldots,a_k,q_1,\ldots,q_s,h$ and relations
\[
\begin{gathered}
 a_iha_i^{-1}=h^{-1},\qquad [q_j,h]=1,\qquad
 q_j^{\alpha_j}h^{\beta_j}=1,\\
 q_1\cdots q_s a_1^2\cdots a_k^2=h^b.
\end{gathered}
\]
Here $h$ represents a regular fiber, $(\alpha_j,\beta_j)$ are the
exceptional Seifert pairs, and $b$ is the integer Seifert invariant.
Over $\mathbb Q$, abelianization gives $h=q_j=0$ and
$a_1+\cdots+a_k=0$, hence $H_1(Y;\mathbb Q)\cong\mathbb Q^{k-1}$.
Since $Y$ is a rational homology sphere, $k=1$, so the coarse base is
$\mathbb{RP}^2$.

Since $Y$ is oriented, orienting the fibers is equivalent to orienting the
base.  Thus $\widehat Y$ fibers over $S^2$, with each Seifert pair appearing
twice in the lifted fibration~\cite[Proposition~2.1]{Peet}.

Suppose first that $s=0$ and $b\ne0$.  The lifted circle bundle over $S^2$
has Euler number $\pm2b\ne0$.  Hence $H^1(\widehat Y;\mathbb Q)=0$ by the
Gysin sequence.  Thus $\widehat Y$ is a closed oriented Seifert fibered
rational homology $3$-sphere with oriented-base cone count $r=0$.

Now suppose that $s=1$.  Absorbing the residual integer into the genuine
Seifert pair gives $(p,q)$ with $p\ge2$ and $\gcd(p,q)=1$, so $q\ne0$.
The lifted fibration over $S^2(p,p)$ has two copies of this pair and hence
rational Euler number
\[
 e(\widehat Y)=\pm\left(\frac{q}{p}+\frac{q}{p}\right)
 =\pm\frac{2q}{p}\ne0.
\]
Since $H^1_{\mathrm{orb}}(S^2(p,p);\mathbb Q)=0$, the rational orbifold
Gysin sequence gives $H^1(\widehat Y;\mathbb Q)=0$.
Thus $\widehat Y$ is a rational homology sphere with oriented-base cone
count $r=2$.

In either case the fiber-orientation-adapted identity
\[
 p^*g_\varepsilon
 =\widehat\pi^*\widehat h+\varepsilon^2\widehat\eta^2
\]
is exactly a fixed connection-metric family on the oriented Seifert rational
homology sphere $\widehat Y$, with $r\le2$.  Theorem~\ref{thm:main} therefore
provides $\varepsilon_0>0$ such that
$(\widehat Y,p^*g_\varepsilon)$ admits no nonzero $\Ztwo$-harmonic one-form
when $0<\varepsilon<\varepsilon_0$.

Finally, a nonzero $\Ztwo$-harmonic one-form $(Z,I,v)$ on
$(Y,g_\varepsilon)$ would pull back to
\[
 \bigl(p^{-1}(Z),p^*I,p^*v\bigr)
\]
on $(\widehat Y,p^*g_\varepsilon)$.  Indeed, $p$ is a finite local isometry:
pullback preserves the flat sign line, commutes with $d_I$ and $d_I^*$, and
preserves the local $L^2$, $W^{1,2}$, and zero-growth conditions on evenly
covered balls.  The pulled-back form is nonzero.
This contradicts the preceding application of Theorem~\ref{thm:main} and
proves~\textup{(ii)}.
\end{proof}

\end{document}